\documentclass[twoside, 11pt]{article}
\usepackage{amsmath,amssymb,amsthm,mathrsfs}
\usepackage[margin=2.5cm]{geometry} 
\usepackage{lipsum}
\usepackage{titlesec,hyperref}
\usepackage{fancyhdr}
\usepackage[numbers,sort&compress]{natbib}
\usepackage{color}

\fancypagestyle{plain}
{
	\fancyhf{}
	
}

\titleformat{\subsection}{\it}{\thesubsection.\enspace}{1.5pt}{}
\titleformat{\subsubsection}{\it}{\thesubsubsection.\enspace}{1.5pt}{}
\numberwithin{equation}{section}

\newtheorem{thm}{Theorem}[section]
\newtheorem{assume}[thm]{Assume}
\newtheorem{lemma}[thm]{Lemma}
\newtheorem{defin}[thm]{Definition}

\newtheorem{prop}[thm]{Proposition}

\newtheoremstyle{rem}{10pt}{10pt}{\rmfamily}{}{\bfseries}{.}{.5em}{} 
\theoremstyle{rem}

\begin{document}

\title{Orbital stability of solitary waves for the Schr\"{o}dinger-Boussinesq system \hspace{-4mm}}

\author{Yilong Ma$^\dag$ \quad Yamin Xiao$^{\ddag,*}$   \\[10pt]	
	\small {$^\dag $School of Mathematical Sciences, Hebei International Studies University,  }\\
	\small {Shijiazhuang, 051132, P.R. China}\\[5pt]	
	\small {$^\ddag$ School of Mathematical Sciences, Hebei Normal University, }\\
	\small {Shijiazhuang, 050024, P.R.
		China}\\[5pt]
}

\footnotetext{*Corresponding author\\
~~~~~E-mail addresses: \it mylshr123@126.com(Y. Ma), \it xiaoyamin@hebtu.edu.cn(Y. Xiao).}
\date{}

\maketitle

\begin{abstract}
This paper studies the orbital stability of solitary waves for the following Schr\"{o}dinger-Boussinesq system
\begin{equation*}
\begin{cases}
{ \begin{array}{ll}
i\varepsilon_t+\varepsilon_{xx}=n\varepsilon+\gamma |\varepsilon|^2\varepsilon, \\
n_{tt}-n_{xx}+ \alpha n_{xxxx}-\beta(n^2)_{xx}=|\varepsilon|^2_{xx},
 \end{array} }  (t,x)\in \mathbb{R}^2.
\end{cases}
\end{equation*}
By applying the abstract results and detailed spectral analysis, we obtain the orbital stability of solitary waves. The result can be
regarded as an extension of the results of \cite{ F-P,H,W}.
\end{abstract}

\vspace{.2in} {\bf Key words:} Solitary waves, orbital stability, Schr\"{o}dinger-Boussinesq system

\vspace{.2in} {\bf MSC(2010):} 74J35, 35B35, 35Q55 
\section{Introduction}
\quad \quad The generalized  Schr\"{o}dinger-Boussinesq system can be written as:
\begin{equation}\label{SB}
\begin{cases}
	{ \begin{array}{ll}
			i\varepsilon_t+\varepsilon_{xx}=n\varepsilon+\gamma |\varepsilon|^2\varepsilon, \\
			n_{tt}-n_{xx}+\alpha n_{xxxx}-\beta(|n|^{p-1}n)_{xx}=|\varepsilon|^2_{xx},
	\end{array} }  (t,x)\in \mathbb{R}^2,
\end{cases}
\end{equation}
where $p>1$, $\alpha,\beta,\gamma\neq0$ are real parameters, $\varepsilon(t,x)$ is a complex valued and $n(t,x)$ is a real valued function.
This kind of problem appears in the study of interaction of solutions in optics. Linares and Navas \cite{L-N} proved the global well-posedness of initial value problem to the system \eqref{SB} in $H^1(\mathbb{R})\times H^1(\mathbb{R})$ with $\beta>0$, and for $\beta<0$, the small global well-posedness is established. In \cite{H-F}, Hona and Fan considered the existence of exact solitary wave solutions for system \eqref{SB}. Later, the orbital stability and instability of solitary wave solutions for \eqref{SB} have been extensively
studied.

In this paper, we consider the existence and orbital stability of solitary wave solutions for \eqref{SB} with $p=2$ (hereafter referred to as the SB-system)  
\begin{align}\tag{SB}\label{1.1}
	\begin{cases}
		{ \begin{array}{ll}
				i\varepsilon_t+\varepsilon_{xx}-n\varepsilon=\gamma |\varepsilon|^2\varepsilon, \\
				n_{tt}-n_{xx}+\alpha n_{xxxx}-\beta(n^2)_{xx}=|\varepsilon|^2_{xx},
		\end{array} }  (t,x)\in \mathbb{R}^2,
	\end{cases}
\end{align}
where $\alpha,\beta,\gamma>0$. \eqref{1.1} can be rewritten as
\begin{equation}\label{3.1}
\begin{cases}
	{ \begin{array}{ll}
			i\varepsilon_t+\varepsilon_{xx}-n\varepsilon=\gamma|\varepsilon|^2\varepsilon, \\
			n_t=\phi_{xx},\\
			\phi_t=n-\alpha n_{xx}+\beta n^2+|\varepsilon|^2.
	\end{array} }
\end{cases}
\end{equation}
Here, we mean the solitary wave solutions of \eqref{3.1} having the following form
\begin{align}\label{2.1}
	\varepsilon(t,x)=\mathrm{e}^{-i\omega t}\varepsilon_{\omega,v}(x-vt), \ \ n(t,x)=n_{\omega,v}(x-vt) \ \mbox{and} \ \ \phi(t,x)=\phi_{\omega,v}(x-vt),
\end{align}
where $\omega, v, q$ are constants, $\varepsilon_{\omega,v}(x):=\mathrm{e}^{iqx}\hat{\varepsilon}_{\omega,v}(x)$ and $\hat{\varepsilon}_{\omega,v}$, $n_{\omega,v}$ and $\phi_{\omega,v}$ are real function. 

We prove the orbital stability of solitary waves of system \eqref{1.1} within the abstract context of Hamiltonian system developed by Grillakis, Shatah, and Strauss in \cite{G-S-S1,G-S-S2}. Hakkaev \cite{H} proved the stability of solitary waves for system \eqref{1.1} with $\alpha=1$, $\gamma=0$ and $\beta=3$. For $\alpha=1$, $\gamma=\beta=0$, Farah and Pastor \cite{F-P} established the orbital stability of standing wave solutions and, Esfahani and Pastor \cite{E-P} obtained the stability result with replacing $n_{xxxx}$ by $-n_{xxtt}$.
When $\alpha=\beta=\gamma=0$, system \eqref{SB} reduces to the Zakharov system
\begin{align}\label{Zakharov}
	\begin{cases}
		{ \begin{array}{ll}
				i\varepsilon_t+\varepsilon_{xx}- n\varepsilon=0, \\
				n_{tt}-n_{xx}=|\varepsilon|^2_{xx},
		\end{array} } (t,x)\in \mathbb{R}^2.
	\end{cases}
\end{align}
Wu \cite{W} proved that the traveling wave solutions of system \eqref{Zakharov} are orbitally stable. Liang \cite{L2011} established the modulational instability to system \eqref{SB}. There are many interesting results concerning the orbital stability and instability of solitary waves for Boussinesq-type and nonlinear Schr\"{o}dinger-type system \cite{C-S,E-L,G-X-B,L-Z,Z-L-L-C}. 
Our work covers the results \cite{W,F-P,H}.

{\bf Notations}. (1) Denote
\begin{align*}
\overrightarrow{u}=(\varepsilon,n,\phi), \ \ \
\overrightarrow{\Phi}_{\omega,v}(x)=\left(\varepsilon_{\omega,v}(x),n_{\omega,v}(x),\phi_{\omega,v}(x)\right).
\end{align*}
(2) Define $X=H^1_{complex}(\mathbb{R})\times H^1_{real}(\mathbb{R})\times L^2_{real}(\mathbb{R})$ with the inner product
\begin{align*}
	 (\overrightarrow{u}_1,\overrightarrow{u}_2)=\mbox{Re}\int_{\mathbb{R}}(\varepsilon_1\overline{\varepsilon}_2+\varepsilon_{1x}\overline{\varepsilon}_{2x}+
	n_1n_2+n_{1x}n_{2x}+\phi_1\phi_2) d x,
\end{align*}
and the dual space of $X$ is $X^*=H^{-1}_{complex}(\mathbb{R})\times H^{-1}_{real}(\mathbb{R})\times L^2_{real}(\mathbb{R})$. There exists a natural isomorphism $I: X\rightarrow X^*$ defined by $\langle I\overrightarrow{u}_1,\overrightarrow{u}_2\rangle=(\overrightarrow{u}_1,\overrightarrow{u}_2)$,
where $\langle\cdot,\cdot\rangle$ denotes the pairing between $X$ and $X^*$, and
\begin{align*}
	 \langle\overrightarrow{f},\overrightarrow{u}\rangle=\mbox{Re}\int_{\mathbb{R}}(f_1\overline{\varepsilon}+f_2n+f_3\phi) d x, \ \ \ I=\begin{pmatrix}&1-\frac{\partial^2}{\partial x^2} \ \ &\ \ \ &\\ & \ \ &1-\frac{\partial^2}{\partial x^2} \ \ &\\ & \ \ & \ \ &1\end{pmatrix}.
\end{align*}
(3) Let $T_1, T_2$ be one-parameter groups of the unitary operator on $X$ defined by
\begin{align}\label{3.8}
	T_1(s_1)\overrightarrow{u}(\cdot)=\overrightarrow{u}(\cdot-s_1), \ T_2(s_2)\overrightarrow{u}(\cdot)=(\mathrm{e}^{-is_2}\varepsilon(\cdot),n(\cdot),\phi(\cdot)), \ \mbox{for} \ \overrightarrow{u}(\cdot)\in X,\ s_1,s_2\in \mathbb{R}.
\end{align}
Obviously,
\begin{align*}
	T_1'(0)=\begin{pmatrix}&-\frac{\partial}{\partial x} \ \ \ &\ \ \ \ & \\ \ & \ \ \ &-\frac{\partial}{\partial x} \ \ \ & \\ \ & \ \ \ \ \ \ \ \ & \ \ \ \ \ \ &-\frac{\partial}{\partial x} \end{pmatrix} \ \ \mbox{and} \ \ T_2'(0)=\begin{pmatrix} &-i \  \ &\ \ \ \ &\ \\ \ & \ \ \ &0 \ \ \ & \\ \ & \ \ \ \ & \ \ \ \ \ \ &0 \end{pmatrix}.
\end{align*}
Note that the system \eqref{3.1} is invariant under $T_1$ and $T_2$ and $T_1(vt)T_2(\omega t)\overrightarrow{\Phi}_{\omega,v}(x)$ is a solitary wave solution of \eqref{3.1}.

In order to present our main results, we recall the definition of the orbital stability of solitary waves $T_1(vt)T_2(\omega t)\overrightarrow{\Phi}_{\omega,v}(x)$.
\begin{defin}[\cite{G-S-S1}]
	The solitary wave $T_1(vt)T_2(\omega t)\overrightarrow{\Phi}_{\omega,v}(x)$ is orbitally stable if for all $\epsilon>0$, there exists $\delta>0$ with the following property: If $\|\overrightarrow{u}_0-\overrightarrow{\Phi}_{\omega,v}\|_X<\delta$ and $\overrightarrow{u}(t)$ is a solution of (\ref{1.1}) in some interval $[0,t_0)$ with $\overrightarrow{u}(0)=\overrightarrow{u}_0$, then $\overrightarrow{u}(t)$ can be continued to a solution in $0\leq t<+\infty$, and
	\begin{align}\label{1.7}
		\sup_{0<t<\infty}\inf_{s_1\in \mathbb{R}}\inf_{s_2\in \mathbb{R}}\|\overrightarrow{u}(t)-T_1(s_1)T_2(s_2)\overrightarrow{ \Phi}_{\omega,v}\|_X<\epsilon.
	\end{align}
	Otherwise, $T_1(vt)T_2(\omega t)\overrightarrow{\Phi}_{\omega,v}$ is called orbitally unstable.
\end{defin}
The paper is organized as follows. In section 2, we introduce the existence of solitary waves and in section 3, we establish some spectral properties we need and prove the orbital stability of solitary waves.
\section{The existence of solitary waves}
\quad \quad In this section, we focus on the existence of solitary waves for \eqref{1.1} system. Putting (\ref{2.1}) into (\ref{3.1}), we get
\begin{align}\label{2.2}
\begin{cases}
{ \begin{array}{ll}
\varepsilon_{\omega,v}''-n_{\omega,v}\varepsilon_{\omega,v}+\omega\varepsilon_{\omega,v}-iv\varepsilon_{\omega,v}'-\gamma|\varepsilon|^2\varepsilon=0,\\
\hat{\varepsilon}_{\omega,v}''+i(2q-v)\hat{\varepsilon}_{\omega,v}'
+(\omega+qv-q^2-n_{\omega,v}-\gamma|\hat{\varepsilon}|^2)\hat{\varepsilon}_{\omega,v}=0,\\
-vn_{\omega,v}=\phi'_{\omega,v},\\
-v\phi'_{\omega,v}=n_{\omega,v}-\alpha n_{\omega,v}''+\beta n_{\omega,v}^2+\hat{\varepsilon}_{\omega,v}^2.
 \end{array} }
\end{cases}
\end{align}
By $(\ref{2.2})_2$ we have $q=\frac{v}{2}$. Let $\hat{\varepsilon}_{\omega,v}=c_1\mbox{sech}c_2x$ satisfy $(\ref{2.2})_2$ with constants $c_1,c_2$ to be determined later, we deduce that
\begin{align}\label{2.3}
\hat{\varepsilon}_{\omega,v}''=(c_2^2-2c_2^2\mbox{sech}^2c_2x)\hat{\varepsilon}_{\omega,v}=(-w-\frac{v^2}{4}+n_{\omega,v}+\gamma|\hat{\varepsilon}_{\omega,v}|^2)\hat{\varepsilon}_{\omega,v}.
\end{align}
Assume $n_{\omega,v}(x)\rightarrow0$ as $x\rightarrow\infty$, one has
\begin{align}\label{2.4}
n_{\omega,v}=-2c_2^2\mbox{sech}^2c_2x-\gamma|\hat{\varepsilon}_{\omega,v}|^2+c_2^2+\omega+\frac{v^2}{4}=-(2c_2^2+\gamma c_1^2)\mbox{sech}^2c_2x,
\end{align}
and
\begin{align}\label{2.5}
c_2^2=-\omega-\frac{v^2}{4}.
\end{align}
Insetting (\ref{2.4})-(\ref{2.5}) into (\ref{2.2}), we obtain that
\begin{align*}
\frac{6\alpha }{\beta}c_2^2=2c_2^2+\gamma c_1^2, \
c_2=\sqrt{-\omega-\frac{v^2}{4}}, \
c_1=\sqrt{\frac{6\alpha c_2^2}{\beta}(1-v^2-4\alpha c_2^2)}.
\end{align*}
In this way, we arrive at the solitary waves solution for system \eqref{3.1}
\begin{equation}\label{2.7}
\begin{cases}
	\hat{\varepsilon}_{\omega,v}=\sqrt{\frac{6\alpha}{\beta} \sigma \eta}\mbox{sech}(\sqrt{\sigma}x),\\
	n_{\omega,v}=-\frac{6\alpha}{\beta}\sigma\mbox{sech}^2(\sqrt{\sigma}x),\\
	\phi_{\omega,v}=\frac{6\alpha}{\beta}v\sqrt{\sigma}\mbox{tanh}(\sqrt{\sigma}x),
\end{cases}
\end{equation}
where $\sigma:=-\omega-\frac{v^2}{4}>0$, $\eta:=1-v^2-4\alpha\sigma>0$. Therefore, we have the following theorem
\begin{thm}\label{thm1.1}
	For any real constants $\omega,v$ satisfying
	\begin{align}\label{2.8}
		-\frac{v^2}{4}-\frac{1-v^2}{4\alpha}<\omega<-\frac{v^2}{4}, \ 1-v^2>0,
	\end{align}
	there exist solitary waves of ({\ref{1.1}}) in the form of (\ref{2.1}), with $\hat{\varepsilon}_{\omega,v}$ and $n_{\omega,v}$ satisfying (\ref{2.7}).
\end{thm}
\section{Orbital stability of solitary waves}
\quad \quad In this section, we prove the orbital stability of solitary wave solutions for \eqref{1.1} system. Define
\begin{align*}
E(\overrightarrow{u})=\int|\varepsilon_x|^2+n|\varepsilon|^2+\frac{\gamma}{2}|\varepsilon|^4
+\frac{1}{2}n^2+\frac{\alpha}{2}|n_x|^2+\frac{\beta}{3}n^3+\frac{1}{2}|\phi_x|^2 d x,
\end{align*}
then the system \eqref{3.1} can be written as the Hamiltonian form
\begin{align}\label{3.3}
\frac{ d \overrightarrow{u}}{ d t}=JE'(\overrightarrow{u}),
\end{align}
where $J$ is a skew symmetric linear operator defined by
\begin{align*}
J=\begin{pmatrix}&-\frac{i}{2} \ \ &0\ \ \ &0\ \ \\ &0 \ \ \ &0\ \ \ &-1\\ &0 \ \ &1 \ \ &0 \end{pmatrix}
\end{align*}
and
\begin{align*}
E'(\overrightarrow{u})=\begin{pmatrix}-2\varepsilon_{xx}+2n\varepsilon+2\gamma|\varepsilon|^2\varepsilon\\ |\varepsilon|^2+n-\alpha n_{xx}+\beta n^2\\-\phi_{xx}\end{pmatrix}
\end{align*}
is the Frechet derivatives of $E(\overrightarrow{u})$.
As in \cite{G-S-S1,G-S-S2}, since $T_1'(0)=J B_1$, $T_2'(0)=J B_2$, one gets $$B_1=\begin{pmatrix}&-2i\frac{\partial}{\partial x} \ &0\ \ &0 \\ &0 \ \ &0 \ \ \ &-\frac{\partial}{\partial x}\\ &0 \ \ &\frac{\partial}{\partial x} \ \ &0\end{pmatrix} \  \mbox{and} \ \ B_2=\begin{pmatrix}&2 \ \ \ & \ \ & \\ & \ \ \ &0 \ \ & \\ & \ \ \ & \ \ &0 \end{pmatrix}.$$
Let
\begin{align*}
Q_1(\overrightarrow{u})=\frac{1}{2}\langle B_1\overrightarrow{u},\overrightarrow{u}\rangle=\frac{1}{2}\int -\phi'n+\phi n'-2i\varepsilon'\varepsilon d x=-\int \phi'n d x+Im\int \varepsilon'\overline{\varepsilon} d x,
\end{align*}
\begin{align*}
Q_2(\overrightarrow{u})=\frac{1}{2}\langle B_2\overrightarrow{u},\overrightarrow{u}\rangle=\int |\varepsilon|^2 d x.
\end{align*}
It is easy to verify that $E(\overrightarrow{u})$, $Q_1(\overrightarrow{u})$ and $Q_2(\overrightarrow{u})$ are invariant under $T_1$ and $T_2$, namely
\begin{align*}
E(T_1(s_1)T_2(s_2)\overrightarrow{u})=E(\overrightarrow{u}),\ \ Q_1(T_1(s_1)T_2(s_2)\overrightarrow{u})=Q_1(\overrightarrow{u}),\ \ Q_2(T_1(s_1)T_2(s_2)\overrightarrow{u})=Q_2(\overrightarrow{u})
\end{align*}
for any $s_1,s_2\in\mathbb{R}$. Moreover, for any $t\in\mathbb{R}$, they are formally conserved under the flow of (\ref{3.1}), that is
\begin{align}
E(\overrightarrow{u}(t))=E(\overrightarrow{u}(0)), Q_1(\overrightarrow{u}(t))=Q_1(\overrightarrow{u}(0)), \ Q_2(\overrightarrow{u}(t))=Q_2(\overrightarrow{u}(0)).
\end{align}
By \eqref{2.2}, we deduce that $\overrightarrow{\Phi}_{\omega,v}$ is a critical point of function $E-vQ_1-\omega Q_2$,
\begin{align}\label{3.15}
E'(\overrightarrow{\Phi}_{\omega,v})-vQ_1'(\overrightarrow{\Phi}_{\omega,v})-\omega Q_2'(\overrightarrow{\Phi}_{\omega,v})=0,
\end{align}
where $Q_1'(\overrightarrow{u})=\begin{pmatrix}-2i\varepsilon'\\ -\phi' \\n'\end{pmatrix}$ and $Q_2'(\overrightarrow{u})=\begin{pmatrix}2\varepsilon\\ 0 \\ 0\end{pmatrix}$. Now, we define the operator $H_{\omega,v}: X\rightarrow X^*$:
\begin{align*}
\nonumber H_{\omega,v}&=E''(\overrightarrow{ \Phi}_{\omega,v})-vQ_1''(\overrightarrow{ \Phi}_{\omega,v})-\omega Q_2''(\overrightarrow{ \Phi}_{\omega,v})\\
&=\begin{pmatrix}&-2\partial_{xx}+2n+2\gamma(|\varepsilon|^2+2\varepsilon^2)+2iv\partial_{x}-2\omega\ \ &2\varepsilon \ \ &0 \\ &2\varepsilon  \  &1-\alpha\partial_{xx}+2\beta n  \ \ &v\partial_x \\ &0 \ \ &-v\partial_x \ \ &-\partial_{xx} \end{pmatrix},
\end{align*}
and we get that
\begin{align}\label{3.17}
H_{\omega,v}\overrightarrow{ \psi}=\begin{pmatrix}-2\psi_{1xx}+2n\psi_1+2\gamma(|\varepsilon|^2\psi_1+\varepsilon\overline{\psi}_1\varepsilon+\overline{\varepsilon}\psi_1\varepsilon) +2iv\psi_{1x}-2\omega\psi_1+2\varepsilon\psi_2\\ \overline{\varepsilon}\psi_1+\varepsilon\overline{\psi}_1+2\beta n\psi_2-\alpha\psi_{2xx}+\psi_2+v\psi_{3x}\\ -v\psi_{2x}-\psi_{3xx}\end{pmatrix}
\end{align}
with $\overrightarrow{\psi}=(\psi_1,\psi_2,\psi_3)\in X$. Observe that $H_{\omega,v}$ is self-adjoint operator. The spectrum of $H_{\omega,v}$ consists of the real numbers $\lambda$ such that $H_{\omega,v}-\lambda I$ is not invertible. Concerning the operator $H_{\omega,v}$ we have the following spectral properties.
\begin{prop}\label{prop}
 (1) Zero is a simple isolated eigenvalue of $H_{\omega,v}$ with $T_1'(0)\overrightarrow{ \Phi}_{\omega,v}(x)$ and $T_2'(0)\overrightarrow{ \Phi}_{\omega,v}(x)$ as its eigenfunction.

(2) On the negative axis $(-\infty,0)$, the spectrum set of $H_{\omega,v}$ admits nothing but only one simple eigenvalue.

(3) Essential spectrum $\sigma_{ess}(H_{\omega,v})$ lies on the positive real axis, admitting a positive distance to the origin.
\end{prop}
\begin{proof}
First, by \eqref{3.8}, \eqref{2.3}-\eqref{2.5} and \eqref{3.19}, it holds that
\begin{align}\label{3.18}
H_{\omega,v}T_1'(0)\overrightarrow{ \Phi}_{\omega,v}(x)=0 \ \ \mbox{and} \ \
H_{\omega,v}T_2'(0)\overrightarrow{ \Phi}_{\omega,v}(x)=0.
\end{align}
Let
\begin{align}\label{3.19}
Z=\{kT_1'(0)\overrightarrow{ \Phi}_{\omega,v}(x)+k_2T_2'(0)\overrightarrow{ \Phi}_{\omega,v}(x) \big| k_1,k_2\in R\},
\end{align}
then by \eqref{3.18}, $Z$ is contained in the kernel of $H_{\omega,v}$.

Second, for any $\overrightarrow{\psi}\in X$, rewrite it as
\begin{align}\label{3.25}
\overrightarrow{\psi}=(\mathrm{e}^{i\frac{v}{2}x}z_1,z_2,z_3)
\end{align}
with $z_1=y_1+iy_2,\ y_1=\mbox{Re} z_1, \ y_2=\mbox{Im} z_1$,
then by \eqref{3.17} we have
\begin{align}\label{3.26}
\nonumber\langle H_{\omega,v}\overrightarrow{\psi},\overrightarrow{\psi}\rangle
=&\int-2\left(\frac{\partial^2}{\partial x^2}-\sigma-n\right)z_1^2+2\gamma\left(|\varepsilon|^2z_1^2+\hat{\varepsilon}^2z_1^2+\hat{\varepsilon}^2\overline{z}_1^2\right)\\
\nonumber&+2\hat{\varepsilon}z_2\overline{z}_1+2\hat{\varepsilon}z_2z_1
+\left(-\alpha\frac{\partial^2}{\partial x^2}+1+2\beta n\right)z_2^2+2vz_3'z_2+(z_3')^2 d x\\
\nonumber=&\int-2\left(\frac{\partial^2}{\partial x^2}-\sigma-n-3\gamma\hat{\varepsilon}^2\right)y_1^2 d x+\int-2\left(\frac{\partial^2}{\partial x^2}-\sigma-n-\gamma\hat{\varepsilon}^2\right)y_2^2 d x\\
\nonumber&+4\int2\hat{\varepsilon}z_2y_1 d x
+\left(-\alpha\frac{\partial^2}{\partial x^2}+1+2\beta n-v^2\right)z_2^2 d x+\int(vz_2+z_3')^2 d x\\
\nonumber=&\langle \overline{L}_1y_1,y_1\rangle+\langle L_2y_2,y_2\rangle+4\int2\hat{\varepsilon}z_2y_1 d x+\langle \overline{L}_3z_2,z_2\rangle+\int(vz_2+z_3')^2 d x\\
\nonumber=&\langle L_1y_1,y_1\rangle+\langle L_2y_2,y_2\rangle+\langle L_3z_2,z_2\rangle+\int(vz_2+z_3')^2 d x\\
&+\int\left(\frac{2\hat{\varepsilon}y_1}{\sqrt{\eta}}+\sqrt{\eta}z_2\right)^2 d x,
\end{align}
where $\overline{L}_1=-2\left(\frac{\partial^2}{\partial x^2}-\sigma-n-3\gamma\hat{\varepsilon}^2\right)$, $L_1=\overline{L}_1-\frac{4\hat{\varepsilon}^2}{\eta}=-2\frac{\partial^2}{\partial x^2}+2\sigma+2n+(6\gamma-\frac{4}{\eta})\hat{\varepsilon}^2$,
$L_2=-2\frac{\partial^2}{\partial x^2}+2\sigma+2n+2\gamma\hat{\varepsilon}^2$,
$\overline{L}_3=-\alpha\frac{\partial^2}{\partial x^2}+1+2\beta n-v^2$ and
$ L_3=\overline{L}_3-\eta=-\alpha\frac{\partial^2}{\partial x^2}+4\alpha\sigma+2\beta n$.
Combining \eqref{2.2} and \eqref{2.7}, one gets
\begin{align}\label{3.28}
 L_{1}\hat{\varepsilon}_x=0, \ \ L_{2}\hat{\varepsilon}=0 \ \ \mbox{and} \ \ L_{3}n_x=0,
\end{align}
and $\hat{\varepsilon}_x$ and $n_x$ have a simple zero point ar $x=0$, respectively. Thanks to the Sturm-Liouville theorem in \cite{H}, we know that $0$ is the second eigenvalue of $L_1$ and $L_3$. Thus $L_1$ and $L_3$ only have one strictly negative eigenvalue with an eigenfunction, respectively. In fact, by the simple calculation, we have
\begin{align}\label{3.29}
L_1(\hat{\varepsilon}^2)=-6\sigma(\hat{\varepsilon}^2) \ \ \mbox{and} \ \ L_3(2n\hat{\varepsilon})=-5\alpha\sigma(2n\hat{\varepsilon}).
\end{align}
We rewrite $L_1=-2\frac{\partial^2}{\partial x^2}+2\sigma+M_1(x)$,
$L_3=-\alpha\frac{\partial^2}{\partial x^2}+4\alpha\sigma+M_3(x)$
with $M_1(x)=2n+(6\gamma-\frac{4}{\eta})\hat{\varepsilon}^2\rightarrow 0$ and $M_3(x)=2\beta n\rightarrow 0$, as $ |x|\rightarrow+\infty$, then by Weyl's theorem on the essential spectrum, it follows that
\begin{align*}
\sigma_{ess}(L_1)=\left[2\sigma,+\infty\right) \ \mbox{and} \
\sigma_{ess}(L_3)=\left[4\alpha\sigma,+\infty\right), \  \ \sigma>0, \ \alpha>0.
\end{align*}
By \eqref{2.7} and \eqref{3.28}, $\hat{\varepsilon}$ has no zero point and $0$ is the first simple eigenvalue of $L_2$ with a eigenfunction $\hat{\varepsilon}(x)$. Rewrite $L_2$ as $L_2=-2\frac{\partial^2}{\partial x^2}+2\sigma+M_2(x)$ with $M_2(x)=2n+2\gamma\hat{\varepsilon}^2$, then $M_2(x)\rightarrow 0$ when $|x|\rightarrow+\infty$, thus
\begin{align*}
\sigma_{ess}(L_2)=\left[2\sigma,+\infty\right),\ \sigma>0
\end{align*}
and we complete the proof of Proposition \ref{prop}.
\end{proof}
Let us state the assumption on the spectral decomposition of $H_{\omega,v}$, which will be used to prove the orbital stability of solitary waves.
\begin{assume}\label{Assum1.1}[\cite{G-S-S1} {\bf Spectral decomposition of $H_{\omega,v}$}]
The space $X$ is decomposed as a direct sum
\begin{align}\label{3.20}
X=N+Z+P,
\end{align}
where $Z$ is defined in \eqref{3.19}, $N$ is a finite-dimensional subspace such that
\begin{align}\label{3.21}
\langle H_{\omega,v} \overrightarrow{u},\overrightarrow{u}\rangle<0, \ \mbox{for} \ 0\neq \overrightarrow{u}\in N,
\end{align}
and $P$ is a closed subspace such that
\begin{align}\label{3.22}
\langle H_{\omega,v} \overrightarrow{u},\overrightarrow{u}\rangle\geq \delta\|\overrightarrow{u}\|_X^2, \ \mbox{for} \ \overrightarrow{u}\in P
\end{align}
with some constant $\delta>0$ independent of $\overrightarrow{u}$.
\end{assume}

According to \cite{A-B}, we have the following lemmas.
\begin{lemma}\label{lem3.1}
For any real functions $y_1\in H^1({\mathbb{R}})$ satisfying
\begin{align*}
\langle y_1,\hat{\varepsilon}^2\rangle=\langle y_1,\hat{\varepsilon}_x\rangle=0,
\end{align*}
then there exists a positive number $\overline{\delta}_1,\delta_1>0$ such that $\langle L_1y_1,y_1\rangle\geq\overline{\delta}_1\|y_1\|_{L^2}^2$. Moreover, we have
\begin{align*}
\langle L_1y_1,y_1\rangle\geq\delta_1\|y_1\|_{H^1}^2.
\end{align*}
\end{lemma}
\begin{lemma}\label{lem3.2}
For any real functions $y_2\in H^1({\mathbb{R}})$ satisfying
\begin{align*}
\langle y_2,\hat{\varepsilon}\rangle=0,
\end{align*}
then there exists a positive number $\delta_2>0$ such that
\begin{align*}
\langle L_2y_2,y_2\rangle\geq\delta_2\|y_2\|_{H^1}^2.
\end{align*}
\end{lemma}
\begin{lemma}\label{lem3.3}
For any real functions $z_2\in H^1({\mathbb{R}})$ satisfying
\begin{align*}
\langle z_2,2n\hat{\varepsilon}\rangle=\langle z_2,n_x\rangle=0,
\end{align*}
then there exists a positive number $\overline{\delta}_3,\delta_3>0$ such that $\langle L_3z_2,z_2\rangle\geq\overline{\delta}_3\|z_2\|_{L^2}^2$. Moreover, we have
\begin{align*}
\langle L_3z_2,z_2\rangle\geq\delta_3\|z_2\|_{H^1}^2.
\end{align*}
\end{lemma}
For any $\overrightarrow{\psi}\in X$, from \eqref{3.25}, we can simply denote $\overrightarrow{\psi}$ by $\overrightarrow{\psi}=(y_1,y_2,z_2,z_3)$.
Choose $y_1^{-}=\hat{\varepsilon}^2$, $y_2^{-}=0$, $z_2^{-}=2n\hat{\varepsilon}$ and $z_{3x}^{-}=-vz_2^{-}$, $\overrightarrow{\psi}^{-}=(y_1^{-},y_2^{-},z_2^{-},z_3^{-})$, then by \eqref{3.26} and \eqref{3.29}, we obtain
\begin{align*}
\langle H_{\omega,v}\overrightarrow{\psi}^{-},\overrightarrow{\psi}^{-}\rangle
=-6\sigma\langle\hat{\varepsilon}^2,\hat{\varepsilon}^2\rangle-
5\alpha\sigma\langle2n\hat{\varepsilon},2n\hat{\varepsilon}\rangle<0.
\end{align*}
Without loss of generality, let
\begin{align*}
N=\{k\overrightarrow{\psi}^{-}/k\in\mathbb{R}\},
\end{align*}
then it implies \eqref{3.21}. In addition, one gets that the kernel of $H_{\omega,v}$ is spanned by the following two vectors
\begin{align*}
\overrightarrow{\psi}_{0,1}=(\hat{\varepsilon}_x,0,n_x,\phi_x) \ \ \mbox{and} \ \
\overrightarrow{\psi}_{0,2}=(0,\hat{\varepsilon},0,0),
\end{align*}
then $Z$ can be rewritten as
\begin{align*}
Z=\{k_1\overrightarrow{\psi}_{0,1}+k_2\overrightarrow{\psi}_{0,2}/k_1,k_2\in\mathbb{R}\}.
\end{align*}
Choosing
\begin{align*}
P=\{\overrightarrow{p}\in X/\overrightarrow{p}=(p_1,p_2,p_3,p_4)\},
\end{align*}
with $\langle p_1,\hat{\varepsilon}^2\rangle+\langle p_3,2n\hat{\varepsilon}\rangle=0$, $\langle p_1,\hat{\varepsilon}_x\rangle+\langle p_3,n_x\rangle=0$ and $\langle p_2,\hat{\varepsilon}\rangle=0$. We will prove that \eqref{3.20} holds. Namely, for any $\overrightarrow{u}=(y_1,y_2,z_2,z_3)\in X$, we choose $$a_1=\frac{\langle y_1,\hat{\varepsilon}^2\rangle+\langle z_2,2n\hat{\varepsilon}\rangle}{\langle \hat{\varepsilon}^2,\hat{\varepsilon}^2\rangle+\langle 2n\hat{\varepsilon},2n\hat{\varepsilon}\rangle},\ b_1=\frac{\langle y_1,\hat{\varepsilon}_x\rangle+\langle z_2,n_x\rangle}{\langle \hat{\varepsilon}_x,\hat{\varepsilon}_x\rangle+\langle n_x,n_x\rangle} \ \mbox{and} \ b_2=\frac{\langle y_2,\hat{\varepsilon}\rangle}{\langle \hat{\varepsilon},\hat{\varepsilon}\rangle}$$ such that $\overrightarrow{u}$ can be uniquely represented by
\begin{align*}
\overrightarrow{u}=a_1\overrightarrow{\psi}^{-}+b_1\overrightarrow{\psi}_{0,1}+b_2\overrightarrow{\psi}_{0,2}+\overrightarrow{p},
\end{align*}
with $\overrightarrow{p}\in P$. For subspace $P$, we have the following lemma, which implies that \eqref{3.22} holds.
\begin{lemma}\label{lem3.4}
For any $\overrightarrow{p}\in P$, there exists a constant $\delta>0$ such that
\begin{align}\label{3.46}
\langle H_{\omega,v}\overrightarrow{p},\overrightarrow{p}\rangle\geq \delta\|\overrightarrow{p}\|_{X}^2
\end{align}
with $\delta$ independent of $\overrightarrow{p}$.
\end{lemma}
\begin{proof}
For any $\overrightarrow{p}\in P$, by \eqref{3.26} and Lemmas \ref{lem3.1}-\ref{lem3.3}, we have
\begin{align}\label{3.47}
\nonumber\langle H_{\omega,v}\overrightarrow{p},\overrightarrow{p}\rangle\geq&\delta_1\|p_1\|_{H^1}^2
+\delta_2\|p_2\|_{H^1}^2+\delta_3\|p_3\|_{H^1}^2+\int(vp_3+p_{4x})^2 d x\\
&+\int\left(\frac{2\hat{\varepsilon}p_1}{\sqrt{\eta}}+\sqrt{\eta}p_3\right)^2 d x.
\end{align}
(1) If $\|p_{4x}\|_{L^2}\geq2|v|\|p_3\|_{L^2}$, then
\begin{align}\label{3.48}
\int(vp_3+p_{4x})^2 d x\geq\frac{1}{4}\|p_{4x}\|_{L^2}^2\geq\delta\|p_4\|_{H^1}^2.
\end{align}
(2) If $\|p_{4x}\|_{L^2}\leq2|v|\|p_3\|_{L^2}$, then
\begin{align}\label{3.49}
\delta_3\|p_3\|_{L^2}^2\geq\frac{\delta_3}{2}\|p_3\|_{L^2}^2+\frac{\delta_3}{4|v|}\|p_{4x}\|_{L^2}^2
\geq\frac{\delta_3}{2}\|p_3\|_{L^2}^2+\frac{\delta\delta_3}{4|v|}\|p_{4}\|_{H^1}^2.
\end{align}
Thus, for any $\overrightarrow{p}\in P$, it follows from \eqref{3.48}-\eqref{3.49} that
\begin{align}\label{3.50}
\nonumber\langle H_{\omega,v}\overrightarrow{p},\overrightarrow{p}\rangle\geq&\delta_1\|p_1\|_{H^1}^2
+\overline{\delta}_2\|p_2\|_{H^1}^2+\overline{\delta}_3\|p_3\|_{H^1}^2+\overline{\delta}_4\|p_4\|_{H^1}^2\\
&+\int\left(\frac{2\hat{\varepsilon}p_1}{\sqrt{\eta}}+\sqrt{\eta}p_3\right)^2 d x.
\end{align}
(1) If $\|p_3\|_{L^2}\geq 2M\|p_1\|_{L^2}$, $M=\frac{2\|\hat{\varepsilon}\|_{L^{\infty}}}{\eta}$, then
\begin{align}\label{3.51}
\int_{\mathbb{R}}\left(\frac{2\hat{\varepsilon} p_1}{\sqrt{\eta}}
+\sqrt{\eta}p_3\right)^2 d x\geq\frac{\eta}{4}\|p_3\|_{L^2}^2.
\end{align}
(2) If $\|p_3\|_{L^2}\leq 2M\|p_1\|_{L^2}$, then
\begin{align}\label{3.52}
\delta_1\|p_1\|^2_{H^1}\geq\frac{\delta_1}{2}\|p_1\|^2_{H^1}+\frac{\delta_1}{8M^2}\|p_3\|_{L^2}^2.
\end{align}
Thus, for any $\overrightarrow{p}\in P$, we arrive at
\begin{align}\label{3.53}
\langle H_{\omega,v}\overrightarrow{p},\overrightarrow{p}\rangle\geq&\delta\|\overrightarrow{p}\|_{X}^2,
\end{align}
and the proof of Lemma \ref{lem3.4} is finished.
\end{proof}
Therefore, the Assumption \ref{Assum1.1} holds. Denote the function $d(\omega,v):\mathbb{R}\times\mathbb{R}\rightarrow\mathbb{R}$ by
\begin{align*}
d(\omega,v)=E(\overrightarrow{\Phi}_{\omega,v})-vQ_1(\overrightarrow{\Phi}_{\omega,v})-\omega Q_2(\overrightarrow{\Phi}_{\omega,v}),
\end{align*}
and $d''(\omega,v)$ is the Hessian matrix of the function $d(\omega,v)$. Along the lines of proofs in [\cite{G-S-S2}, section 3 and 4], we shall apply the abstract stability theory to prove the orbital stability of solitary waves $T_1(vt)T_2(\omega t)\overrightarrow{\Phi}_{\omega,v}(x)$ to system \eqref{3.1}.
\begin{thm}\label{thm2.1}[{\bf Abstract Stability Theorem} \cite{G-S-S1,G-S-S2}]
Assume that there exists the solitary wave solutions $T_1(vt)T_2(\omega t)\overrightarrow{\Phi}_{\omega,v}(x)$ of \eqref{3.1} and Assumption \ref{Assum1.1} holds. Let $n(H_{\omega,v})$ be the number of negative eigenvalues of $H_{\omega,v}$. Assume $d(\omega,v)$ is non-degenerate at $(\omega,v)$ and let $p(d'')$ be the number of positive eigenvalue of $d''$. If $p(d'')=n(H_{\omega,v})$, then the solitary wave solutions $T_1(vt)T_2(\omega t)\overrightarrow{\Phi}_{\omega,v}(x)$ is orbitally stable.
\end{thm}
Here is a more detailed results about stability of solitary wave $T_1(vt)T_2(\omega t)\overrightarrow{\Phi}_{\omega,v}(x)$ in this paper.
\begin{thm}\label{thm3.1}
Under the condition of the Theorem \ref{thm1.1}, the solitary waves $T_{1}(vt)T_2(\omega t)\overrightarrow{\Phi}_{\omega,v}(x)$ of system \eqref{3.1} are orbital stable if
\begin{align}\label{3.24}
-\frac{v^2}{4}-\frac{1+5v^2}{2\alpha(6+\beta)}<\omega<-\frac{v^2}{4}.
\end{align}
\end{thm}
\begin{proof} To accomplish the proof of Theorem \ref{thm3.1}, we need to prove $n(H_{\omega,v})=p(d'')=1$. By the Proposition \ref{prop} and Assumption \ref{Assum1.1}, we just have to prove that $p(d'')=1$. It can be deduced from \eqref{3.15} that
\begin{align*}
d_{\omega}(\omega,v)=&-Q_2(\Phi_{\omega,v})
=-\int_{\mathbb{R}}\varepsilon\overline{\varepsilon} d x
=-\frac{12\alpha }{\beta}\sqrt{\sigma}\eta,\\
d_{v}(\omega,v)=&-Q_1(\Phi_{\omega,v})=\int_{\mathbb{R}}\phi'n d x-\mathrm{Im}\int_{\mathbb{R}}\varepsilon'\overline{\varepsilon} d x
=-\frac{48\alpha^2}{\beta^2}v\sigma^{\frac{3}{2}}-\frac{6\alpha }{\beta}v\sqrt{\sigma}\eta.
\end{align*}
Hence, by the simple calculations, we obtain
\begin{align*}
d_{\omega\omega}(\omega,v)=&\frac{6\alpha }{\beta}\frac{\eta-8\alpha\sigma}{\sqrt{\sigma}},\\
d_{\omega v}(\omega,v)
=&\frac{3\alpha }{\beta}\frac{v\eta}{\sqrt{\sigma}}-\frac{24\alpha }{\beta}v(\alpha-1)\sqrt{\sigma},\\
d_{v\omega}(\omega,v)
=&\frac{3\alpha }{\beta}\frac{v\eta}{\sqrt{\sigma}}-\frac{24\alpha }{\beta}v(\alpha-\frac{3\alpha }{\beta})\sqrt{\sigma},\\
d_{vv}(\omega,v)=&-\frac{48\alpha^2 }{\beta^2}\sigma^{\frac{3}{2}}
+\frac{36\alpha^2 }{\beta^2}v^2\sqrt{\sigma}-\frac{6\alpha }{\beta}\sqrt{\sigma}\eta
-\frac{12\alpha }{\beta}v^2(\alpha-1)\sqrt{\sigma}
+\frac{3\alpha }{\beta}\frac{v^2}{2}\frac{\eta}{\sqrt{\sigma}}
\end{align*}
and
\begin{align*}
\mathrm{det}(d'')=&d_{\omega\omega}d_{vv}-d_{\omega v}d_{v\omega }\\
=&-288\frac{\alpha^3}{\beta^3}\eta\sigma-36\frac{\alpha^2}{\beta^2}\eta^2+48\cdot48\frac{\alpha^3}{\beta^3}\alpha\sigma^2+288\frac{\alpha^2}{\beta^2}\alpha\sigma \eta-24\cdot72\frac{\alpha^3}{\beta^3}v^2\sigma\\
=&-36\frac{\alpha^2}{\beta^2}\left((\eta-4\alpha \sigma)^2+8\frac{\alpha}{\beta}\sigma(1+5v^2-12\alpha\sigma-2\beta\alpha\sigma)\right)\\
<&0
\end{align*}
provided $-\frac{v^2}{4}-\frac{1+5v^2}{2\alpha(6+\beta)}<\omega$. It follows that $d''$ has exactly one positive and one negative eigenvalue, thus $p(d'')=1$ and we complete the proof of Theorem \ref{thm3.1}.
\end{proof}
\textbf{Acknowledgements} We would like to thank the referees very much for their valuable comments and suggestions.


\end{document}